\documentclass[12pt]{article}

\usepackage{amssymb,amsmath,amsthm,sectsty,url,mathrsfs,graphicx}
\usepackage[letterpaper,hmargin=1.0in,vmargin=1.0in]{geometry}
\usepackage[dvips,colorlinks,linkcolor=black,citecolor=black,filecolor=black,urlcolor=black]{hyperref}
\usepackage{color}

\sectionfont{\large}
\subsectionfont{\normalsize}
\numberwithin{equation}{section}

\newtheorem{theorem}{Theorem}[section]
\newtheorem{lemma}[theorem]{Lemma}
\newtheorem{proposition}[theorem]{Proposition}

\newtheorem{definition}[theorem]{Definition}
\newtheorem{remark}[theorem]{Remark}

\renewcommand{\Pr}{ \mathrm P}

\newcommand{ \rel}{ t_{\mathrm{rel}} }
\newcommand{ \mix}{ t_{\mathrm{mix}} }

\newcommand{ \hit}{ \mathrm{hit} }

\newcommand{ \TV}{ \mathrm{TV} }

\newcommand{\eps}{\epsilon}

\newcommand{\la}{\lambda}

\DeclareMathSymbol{\leqslant}{\mathalpha}{AMSa}{"36} 
\DeclareMathSymbol{\geqslant}{\mathalpha}{AMSa}{"3E} 
\DeclareMathSymbol{\eset}{\mathalpha}{AMSb}{"3F}     

\renewcommand{\epsilon}{\varepsilon}

\newcommand{\N}{\mathbb N}
\newcommand{\R}{\mathbb R}
\newcommand{\Z}{\mathbb Z}

\usepackage[normalem]{ulem}

\begin{document}

\title{Cutoff for Ramanujan graphs via degree inflation}
\author{Jonathan Hermon
\thanks{
Faculty of mathematics and computer science, Weizmann Institute of Science, Rehovot, Israel. E-mail: {\tt jonathah@weizmann.ac.il}.}}
\date{}
\maketitle

\begin{abstract}
Recently Lubetzky and Peres showed that simple random walks on a sequence of $d$-regular Ramanujan graphs $G_n=(V_n,E_n)$ of increasing sizes exhibit cutoff in total variation around the diameter lower bound $\frac{d}{d-2}\log_{d-1}|V_n| $. We provide a different argument  under the assumption that for some $r(n) \gg 1$ the maximal number of simple cycles in a ball of radius $r(n)$ in $G_n$ is uniformly bounded in $n$.
\end{abstract}


\section{Introduction} 

Generically, we denote the stationary distribution of an ergodic Markov chain $(X_t)_{t \ge 0}$ by $\pi$,  its state space by $\Omega $ and its transition matrix by $P$. We denote by $\Pr_{x}^t$ (resp.~$\Pr_{x}$) the distribution of $X_t$ (resp.~$(X_t)_{t \ge 0 }$), given that the initial state is $x$. The total variation distance of two distributions on $\Omega $ is $\|\mu - \nu \|_{\TV}=\frac{1}{2} \sum_y | \mu(y)-\nu(y)|$. The total variation $\epsilon$-mixing time is $\mix(\epsilon):= \inf \{t: \max_x \|\Pr_x^t - \pi \|_{\TV} \le \epsilon \} $.    
Next, consider a sequence of chains, $((\Omega_n,P_n,\pi_n))_{n \in \N}$,
each with its 
mixing  time $t_{\mathrm{mix}}^{(n)}(\cdot)$. We say that the sequence exhibits a cutoff if the
following
sharp transition in its convergence to stationarity occurs:
\begin{equation}\label{def:cutoff}
\forall \eps \in (0,1/2], \quad  \lim_{n \to \infty}t_{\mathrm{mix}}^{(n)}(\epsilon)/t_{\mathrm{mix}}^{(n)}(1-\epsilon)=1.
\end{equation}
A family of $d$-regular graphs $G_n$ with $d \ge 3$ is called an \emph{expander family}, if the second largest eigenvalues of the corresponding adjacency matrices are uniformly bounded away from $d$. Lubotzky, Phillips, and Sarnak \cite{Lub}
defined  a connected finite $d$-regular graph $G$ with $d \ge 3$ to be \textbf{\emph{Ramanujan}} if the eigenvalues of the transition matrix of simple random walk (SRW) on $G$ all lie in $\{\pm 1\} \cup [-\rho_d,\rho_d] $, where $\rho_d:=\frac{2\sqrt{d-1}}{d} $ is the spectral radius of SRW on the infinite $d$-regular tree $\mathbb{T}_d $. Lubotzky, Phillips, and Sarnak \cite{Lub}, Margulis \cite{margulis} and Morgenstern \cite{morg} constructed $d$-regular Ramanujan graphs for all $d$ of the form $d=p^{m}+1$, where $p$ is a prime number. Recently, Marcus,  Spielman and Srivastava \cite{existence} proved the existence of bipartite $d$-regular Ramanujan graphs for all $d \ge 3$.  In light of the Alon-Boppana bound \cite{AB}, Ramanujan graphs are ``optimal expanders" as they have asymptotically the largest spectral-gap.

Let $G_n=(V_n,E_n)$ be a sequence of finite connected $d_n$-regular graphs. Let $P_n$ be the transition matrix of SRW on $G_n$. Denote the eigenvalues of $P_n$ by $1=\la_1(n)>\la_2(n) \ge \cdots \ge \la_{|V_n|}(n) \ge -1$. We say that the sequence is asymptotically Ramanujan if $|V_n| \to \infty $ and \[\max  \{|\la_i(n)|:| \la_i(n) |\neq 1 \} \le \rho_{d_n}^{1-o(1)} .\] We say that the sequence is asymptotically one-sided Ramanujan if $|V_n| \to \infty $, $\la_2(n) \le \rho_{d_n}^{1-o(1)} $ and $\liminf_{n \to \infty}\min \{\la_i(n): \la_i(n) \neq -1 \}>-1$.  Friedman \cite{friedman2} showed that a sequence of $d$-regular random graphs of increasing sizes is w.h.p.~asymptotically Ramanujan.

\begin{remark}
\label{rem:nonstandard}
Our definition of asymptotically Ramanujan graphs is not the standard one. The more standard definition is that  $\max  \{|\la_i(n)|:| \la_i(n) |\neq 1 \} \le \rho_{d_n}+o(1)$.
\end{remark}

It is elementary to show that for every $n$-vertex $d$-regular graph, the $1-\epsilon$ total variation mixing time for the SRW is at least $t_{d,\eps,n}:=\frac{d}{d-2}\log_{d-1}n-C\sqrt{n |\log \epsilon|/d}$, for some constant $C>0$.\footnote{This can be derived from the fact that $C$ can be chosen so that the probability that the probability that the distance of the walk at time $t_{d,\eps,n}$   from its starting point is at least $\lfloor \log_{d-1}(\frac{1}{4} \eps n)\rfloor $ with probability at most $\frac{\eps}{2}$ (together with the fact that a ball of radius $\lfloor \log_{d-1}(\frac{1}{4} \eps n)\rfloor $ contains at most $\frac{1}{2} \eps n$ vertices).} The following precise formulation of this fact is due to Lubeztky and Peres
\cite{ram}. 
\begin{lemma}[Trivial diameter lower bound - c.f.~\cite{ram} (2.2)-(2.3) pg.~9]
\label{lem: lower1}
Let $G=(V,E)$ be an $n$-vertex  $d$-regular graph with $d \ge 3$.  Let $c_d:=\frac{2\sqrt{d(d-1)}}{(d-2)^{3/2}} $ and $\Phi^{-1}$ be the inverse function of the CDF of the standard Normal distribution.  Then SRW on $G$ satisfies  
\[\forall \epsilon \in (0,1), \quad \mix(1- \epsilon -o(1) ) \ge \frac{d}{d-2}\log_{d-1}n+c_{d} \Phi^{-1}(\epsilon)\sqrt{\log_{d-1}n} . \]
\end{lemma}

Recently, Lubetzky and Peres \cite{ram} showed that simple random walks on a sequence of non-bipartite  $d_n$-regular Ramanujan graphs  $G_n=(V_n,E_n)$  of increasing sizes exhibit cutoff  around the diameter lower bound $\frac{d_{n}}{d_{n}-2}\log_{d_{n}-1}|V_n| $. In this work we present an alternative argument and prove the same result under the following assumption:

\medskip

\textbf{Assumption 1}: There exists a diverging sequence $r_{n}$ such that the maximal number of simple cycles in a ball of radius $r_n$ in $G_n$ is uniformly bounded in $n$.
\begin{theorem}
\label{thm:main}
Let $G_n=(V_n,E_n)$ be a sequence of non-bipartite, finite, connected, $d_n$-regular asymptotically one-sided Ramanujan graphs.
\begin{itemize}
\item[(i)] If $d_n=d$ for all $n$ and Assumption 1 holds then the corresponding sequence of simple random walks exhibits cutoff around time $\frac{d}{d-2}\log_{d-1}|V_n| $. 
\item[(ii)] If $d_n $ diverges and $\log d_n =o(\log_{d_n}|V_n|) $ then the corresponding sequence of simple random walks exhibits cutoff around time $\log_{d_{n}}|V_n| $.
\end{itemize}
\end{theorem}
\begin{remark}
\label{rem: subseq} If there is no cutoff, then cutoff must fail on some subsequence $(n_{k})$ such that either $\lim_{k \to \infty}d_{n_{k}}=\infty $ or $d_{n_k}=d $ for all $k$ for some fixed $d \ge 3$. Thus there is no loss of generality in assuming that either $\lim_{n \to \infty}d_{n}=\infty $ or $d_{n}=d $ for all $n$. 
\end{remark}
Assumption 1 is rather mild as it is quite difficult to construct a family of asymptotically one-sided Ramanujan graphs violating this assumption. In particular, it is satisfied w.h.p.~by a sequence of random $d$-regular  graphs of increasing sizes \cite{LS}. It follows from \cite[Theorem 1]{kesten} that if $G_n$ is a sequence of $d$-regular transitive asymptotically Ramanujan graphs of increasing sizes then $\lim_{n \to \infty} \mathrm{girth}(G_n) = \infty $, where for a graph $G$, $\mathrm{girth}(G)$ denotes its girth\footnote{The girth of a graph $G$ is the length of the shortest cycle in $G$.} (and so Assumption 1 holds).  

\medskip

The argument of Lubetzky and Peres \cite{ram} does not require Assumption 1 (nor the assumption $\log d_n =o(\log_{d_n}|V_n|) $). They studied the Jordan decomposition of the transition matrix of the non-backtracking walk\footnote{This is a random walk on the directed edges of the graph, with transition matrix $P_{\mathrm{NB}}((x,y)(z,w))=\frac{1_{z=y,w \neq x}}{\deg(y)-1} $.} and used it to derive cutoff for the non-backtracking walk, which for a regular graph implies cutoff also for the SRW. In this note we study the SRW by looking at it only when it crosses distance $k$ from its previous position, for some large $k$.  

\subsection{Organization of this note}
In \S~\ref{s:2}, as a warm up, we present an extremely simple and short proof for the occurrence of cutoff for SRW on  a sequence of asymptotically  Ramanujan graphs of diverging degree. In \S~\ref{s:3} we present some machinery for bounding mixing times using hitting times. We then apply this machinery to prove Part (ii) of Theorem \ref{thm:main}. In \S~\ref{s:4} we give an overview of the proof of Part (i) of Theorem \ref{thm:main}. In \S~\ref{s:5} we prove two auxiliary results. Finally, in \S~\ref{s:6} we conclude the proof of Theorem \ref{thm:main}.  

\section{A warm up}
\label{s:2}
It turns out that  for a sequence of asymptotically  Ramanujan graphs of diverging degree the trivial diameter lower bound (of Lemma \ref{lem: lower1}) is matched by the  trivial spectral-gap upper bound on the $L_2$ mixing time obtained via the Poincar\'e inequality. As a warm up and motivation for what comes we now prove the following theorem.
\begin{theorem}
\label{thm:easy}
Let $G_n=(V_n,E_n)$ be a sequence of non-bipartite, finite, connected, $d_n$-regular asymptotically  Ramanujan graphs with $d_n \to \infty$. Then the corresponding sequence of simple random walks exhibits cutoff around time $\log_{d_{n}}|V_n| $.
\end{theorem}
Note that in Part (ii) of Theorem \ref{thm:main} the graphs are assumed to be only asymptotically one-sided Ramanujan. Before proving Theorem \ref{thm:easy} we need a few basic definitions and facts. Let \[\la:=\max \{|a|:a \neq 1, a \text{ is an eigenvalue of }P \} \quad \text{and} \quad \rel:=\frac{1}{1-\la}. \] The $L_2$ distance of $\Pr_x^t $ from $\pi$ is defined as  \[ \| \Pr_x^t-\pi \|_{2,\pi}^2=\sum_y \pi(y) ( P^{t}(x,y)/\pi(y))^2-1.\] By Jensen's and the Poincar\'e inequalities, for all $t$ and $x$  we have that \[4 \| \Pr_x^t-\pi \|_{\TV}^2 \le \| \Pr_x^t-\pi \|_{2,\pi}^2 \le \la^{2t}\| \Pr_x^0-\pi \|_{2,\pi}^2 \le \la^{2t}/\pi(x).\] Hence for SRW on an $n$-vertex regular graph we have for all $t$ and $x$ that 
\begin{equation}
\label{eq: poin}
4 \| \Pr_x^t-\pi \|_{\TV}^2 \le  n \la^{2t} \Longrightarrow \mix(\epsilon) \le \frac{1}{2} \log_{\frac{1}{\la}}(n \epsilon^{-2}). \end{equation}

\emph{Proof of Theorem \ref{thm:easy}:}
 By assumption  $\la  = \rho_{d_{n}}^{1-o(1)} =d_{n}^{-\frac{1}{2}(1-o(1))}$.  Thus $   \frac{1}{2} \log_{\frac{1}{\la}}|V_n| =(1+o(1)) \log_{d_{n}}|V_{n}| $. The proof is concluded by combining \eqref{eq: poin} with Lemma \ref{lem: lower1}. \qed 

\section{Replacing the Poincar\'e inequality by its hitting time analog}
\label{s:3}
In the proof of Theorem \ref{thm:main} we exploit the general connection between mixing times and escape times from small sets, established in \cite{cutoff} (Corollary 3.1 eq.~(3.2)): There exists some absolute constant $C>0$ such that for every reversible chain (with a finite state space), \begin{equation}
\label{eq: hitmix}
\forall \,  \alpha, \epsilon \in (0,1), \quad \mix(\epsilon + \alpha) \le \hit_{1-\alpha}(\epsilon)+C \rel \log (1/\alpha),
\end{equation}
where $\hit_{1-\alpha}(\epsilon):=\inf \{t:\max_{x,A: \pi(A) \le \alpha} \Pr_x[T_{A^{c}}>t] \le \eps \}$ and $T_B:=\inf \{t:X_t \in B \}$ is the hitting time of the set $B$. In the proof of Theorem \ref{thm:main} we replace the naive $L_2$ bound used in the proof of Theorem \ref{thm:easy} by its  hitting time counterpart: Under reversibility, for all $A \subsetneq \Omega$,  $a \in A$ and $t \ge 0$
\begin{equation}
\label{eq: spectralhit}
\pi_A(a)(\Pr_a[T_{A^c}>t])^{2} \le    \sum_{b \in A} \pi_A(b)( \Pr_{b}[T_{A^c}>t] )^{2}=   \|P_A^{t} 1_A \|_{2, A}^2 \le [\la(A)]^{2t},
\end{equation}
where $\pi_A$ is $\pi$ conditioned on $A$, $P_A$ is the restriction of the transition matrix $P$ to $A$ (this is the transition matrix of the chain which is ``killed" upon escaping $A$), $\|f\|_{2,A}^2:=\sum_{b \in A} \pi_A(b)f^{2}(b)$ for $f \in \R^{A}$ and $\la(A)$ is the largest eigenvalue of $P_A$. 

The following proposition relates $\la(A)$ to $\la_2$, the second largest eigenvalue of $P$. 

\begin{proposition}[e.g.~\cite{cutoff} Lemma 3.8]
\label{prop: laAla2}
For every reversible Markov chain and any set $A$,  
\begin{equation}
\label{e:la(A)e}
\la(A) \le \la_2+\pi(A), \end{equation}
\end{proposition} 
Similarly to \eqref{eq: poin}, by \eqref{eq: hitmix}-\eqref{e:la(A)e} we have for every reversible chain on a finite state space with $\la_2<1/2$ and every $\alpha \in (0,\la_2]$ that
\begin{equation}
\label{e:la(A)e2}
\begin{split}
& \hit_{1-\alpha}(\sqrt{ \alpha}) \le \frac{1}{2} | \log_{\frac{1}{2 \la_2}}(\min_v \pi(v))|, 
\\ & \mix(2 \sqrt{ \alpha}) \le \frac{1}{2} | \log_{\frac{1}{2 \la_2}}(\min_v \pi(v))|+C \rel \log (1/\alpha).
\end{split}
\end{equation}
We are now in a position to give a short proof for Part (ii) of Theorem \ref{thm:main}.
\begin{proof}
Let $G_n=(V_n,E_n)$ be a sequence of non-bipartite, finite, connected, $d_n$-regular asymptotically one-sided Ramanujan graphs. Assume that $d_n $ diverges and $\log d_n =o(\log_{d_n}|V_n|) $. Let $\alpha=\alpha_n=d_n^{-1/2}=o(1) $. Let $\la_2=\la_2(n)$ be the second largest eigenvalue of the transition matrix of SRW on $G_n$. By our assumptions  $2 \la_2=d_n^{-\frac{1}{2}+o(1)} $
and so by \eqref{e:la(A)e2} we have that \[\mix(2 \sqrt{ \alpha}) \le \frac{1}{2}  \log_{\frac{1}{2 \la_2}}|V_{n}|+C '\log (1/\alpha)=(1+o(1))\log_{d_n}|V_n| . \]
The proof is concluded using Lemma \ref{lem: lower1}.
\end{proof} 
\section{Degree inflation}
\label{s:4}
The simple proof of Part (ii) of Theorem \ref{thm:main} motivates looking at the following graph.

\begin{definition}
\label{d:Gk}
Given a graph $G=(V,E)$, we define $G(k)=(V,E(k))$ via \[E(k):=\{\{u,v \} : \mathrm{dist}_{G}(u,v)=k,\, u,v \in V \},\] where $\mathrm{dist}_{G}(u,v)$ denotes the graph distance of $u$ and $v$ w.r.t.~$G$. Denote the transition matrix of SRW on $G(k)$ by $K$.
\end{definition}
\begin{definition}
\label{d:Y}
 Consider SRW on $G$, $(X_t)_{t=0}^{\infty}$.  Let $T_0:=0$ and inductively set $T_{i+1}:=\inf \{t \ge T_i : \mathrm{dist}_{G}(X_{T_{i+1}},X_{T_{i}})=k \} $. Consider the chain $\mathbf{Y}:=(Y_j)_{j=0}^{\infty}$ defined via $Y_i:=X_{T_i}$ for all $i$, and denote its transition matrix by $W $.
\end{definition}
\begin{remark}
\label{rem:notcon}
It is possible that $G(k):=(V,E(k))$ is not connected. This could be rectified, say by connecting every vertex to its entire $k$-neighborhood. However, below we only use the fact that the SRW on $G(k)$ is reversible w.r.t.~$\pi_{G(k)}(x):=\deg_{G(k)}(x)/(2|E(k)|) $.
\end{remark}
Let $G=(V,E)$ be a $d$-regular finite Ramanujan graph. Assume that Assumption 1 holds. Let $r=r_n$ be as in Assumption 1. Fix some $k=k_n$ such that $1 \ll k \ll \sqrt{ r}$. 
\begin{remark}
\label{rem:www}
Let $K,W$ and $T_i$ be as in Definitions \ref{d:Gk} and \ref{d:Y}.
By Assumption 1, for every $x,y \in V$ of distance $k$ from one another $1 \le K(x,y)d(d-1)^{k-1} \le C_1(d)$. In Lemma \ref{lem:W} we show that for such $x,y$ also $1 \le W(x,y)d(d-1)^{k-1} \le C_2(d)$. In fact, Assumption 1 could have been replaced by the assumption that $\max\{W(x,y),K(x,y) \} \le (d-1)^{-k(1-o(1))} $ and that $T_1$ is concentrated around $\frac{dk}{d-2}$ (uniformly for all initial states).
\end{remark}
\subsection{An overview of the proof of Part (i) of Theorem \ref{thm:main}}
Let $G,k$ and $r$ be as above.
Intuitively,  if either the SRW on $G(k)$ or  the chain $\mathbf{Y}$ (from Definitions \ref{d:Gk} and \ref{d:Y}) exhibit an abrupt convergence to stationarity around time $t=t_n$, then  also the  SRW on $G$ should exhibit an abrupt convergence to stationarity around time $t \cdot \frac{d}{d-2}k $. The term $\frac{d}{d-2}k $ comes from the fact that (by Assumption 1) the expected time it takes the walk on $G$ to get within distance $k  $ from its current position is $\frac{d}{d-2}k(1+o(1))$.

While the chain $\mathbf{Y}$ is more directly related to the SRW on $G$, it is harder to analyze it directly since it need not be reversible and a-priori it is not clear that its stationary distribution is close to the uniform distribution. Instead we analyze the walk on $G(k)$ and use it to learn about $\mathbf{Y}$ and then in turn about the walk on $G$.

\medskip

In light of Part (ii) of Theorem \ref{thm:main} (which has already been proven)  a natural strategy for proving Part (i) of Theorem \ref{thm:main} is to show that  
 $\la_2(K) = \rho_D^{1-o(1)}=(d-1)^{-\frac{k}{2}(1-o(1))}$, where $D$ is the maximal degree in $G(k)$, $K$ is the transition matrix of  SRW on $G(k)$   and $\la_2(K)$ is its second largest eigenvalue. Unfortunately, we do not know how to show this (see the first paragraph of \S~\ref{s:5}). Instead, we obtain such an estimate for $\la_{K}(A)$, the largest eigenvalue of $K_A$, the restriction of $K$ to $A$, for any ``small" set $A$. By small we mean that its stationary probability is at most $\alpha:=(d-1)^{-3k^2} $. 
Indeed,
the key to the proof of Part (i) of Theorem \ref{thm:main} is to show that $\la_{K}(A) \le (d-1)^{-\frac{k}{2}(1-o(1))}$     for every small set $A$. Using \eqref{eq: spectralhit} we get for the walk on $G(k)$ that  $\Pr_a[T_{A^c}>(1+o(1))\frac{1}{k} \log_{d-1}|V| ]=(d-1)^{-\frac{k}{2}(1-o(1))}$. We then show that the same holds for $\mathbf{Y}$ (this is obvious when $2k < \mathrm{girth}(G)$; The general case is derived using the fact that, as mentioned in Remark \ref{rem:www}, $c W(x,y) \le K(x,y) \le CW(x,y) $ for all $x,y$). Finally,  using an obvious coupling between  $\mathbf{Y} $ and the SRW on $G$, after multiplying by $\frac{d}{d-2}k (1+o(1)) $ the last bound is transformed into a bound on  $ \hit_{1-\alpha}(o(1))$ for SRW on $G$ (for some $o(1)$ terms).

\section{Auxiliary results}
 \label{s:5}
    In order to control $\la_{K}(A)$ (for small $A$), apart from Proposition \ref{prop: laAla2} we need the following comparison result. While there are similar comparison techniques for the spectral-gap, we are not aware of a comparison technique which allows one to argue that $\la_2$  (the second largest eigenvalue of the transition matrix) of one chain is close to 0 (say, that $\la_2=o(1)$) if that of another chain is close to 0.
%
\begin{proposition}
\label{prop: comparison}
Let $P^{(1)}$ and $P^{(2)}$ be two transition matrices on the same finite state space $\Omega$, both reversible w.r.t.~$\pi^{(1)}$ and $\pi^{(2)}$, respectively. Assume that $P^{(1)}(x,y) \le C_{1} P^{(2)}(x,y) $ and $1/C_2 \le \pi^{(1)}(x)/\pi^{(2)}(x) \le C_{2}  $ for all $x,y$. Let $A \subsetneq \Omega $ and let $\la_{P^{(i)}}(A) $ be the largest eigenvalue of $P_A^{(i)}$, the restriction of $P^{(i)}$ to $A$ $(i=1,2)$.  Then
\[\la_{P^{(1)}}(A) \le C_{1} C_2^2 \la_{P^{(2)}}(A). \]
\end{proposition}
\emph{Proof:}
Denote $\langle f,g\rangle_{\pi_{A}^{(i)}}:=\sum_{x\in A} \pi_{A}^{(i)}(x)g(x)f(x) $. By the Perron-Frobenius Theorem \[\la_{P^{(1)}}(A)=\max_{f \in \R_+^A,f \neq 0} \frac{ \langle P_{A}^{(1)}f,f\rangle_{\pi_{A}^{(1)}}}{\langle f,f\rangle_{\pi_{A}^{(1)}}} \le C_1C_{2}^2 \max_{f \in \R_+^A,f \neq 0} \frac{ \langle P_{A}^{(2)}f,f\rangle_{\pi_{A}^{(2)}}}{\langle f,f\rangle_{\pi_{A}^{(2)}}} =C_1C_{2}^2 \la_{P^{(2)}}(A). \qed \]
Before proving Theorem \ref{thm:main} we need one more lemma.
\begin{lemma}
\label{lem:W}
Let $G=(V,E)$ be a  $d$-regular graph ($d \ge 3$). Let $v \in V$. For  $i,k \in \N $   let $D_i:=\{u \in V : \mathrm{dist}_{G}(u,v)=i \}$,  $B_{i}:=\cup_{j=0}^{i} D_j$ (the ball of radius $i$ around $v$) and 
\[\mathrm{t}(B_k):=|\{\{x,y\} \in E:y \in B_{k-1},x \in B_k \}|-|B_{k}|. \]
 For any $s \ge 0 $ there exist some  constant $C(s,d)>0$ and $k_s$ such that if $k \ge k_s$, $\mathrm{t}(B_k) \le s $ and $D_k \neq \eset $ then \begin{equation}
\label{e:hitDk}
\frac{1}{d(d-1)^{k-1} } \le  \min_{u \in D_k }\Pr_v[T_{D_{k}}=T_u] \le \max_{u \in D_k }\Pr_v[T_{D_{k}}=T_u] \le \frac{C(s,d)}{d(d-1)^{k-1} }.
\end{equation}  
\end{lemma}
\begin{proof}
 $ $
Let $u \in D_k$. We first prove that $  \Pr_v[T_{D_{k}}=T_u] \ge \frac{1}{d(d-1)^{k-1} } $. This follows from a standard argument involving the covering tree of $G$.  
A non-backtracking path of length $\ell$ is a sequence of vertices $(v_0,v_1,\ldots,v_{\ell})$ such that $\{v_i,v_{i-1} \} \in E $ and $v_{i+2}  \neq v_i $ for all $i$. Let $\mathcal{P}_{\ell} $ be the collection of all non-backing paths of length $\ell$ starting from $v$. Let $\mathbb{T}_d $ be the (infinite) $d$-regular tree. We may label the $\ell$th level of $\mathbb{T}_d $ by the set $\mathcal{P}_{\ell}$ (in a bijective manner) such that the children of  $(v,v_1,\ldots,v_{\ell})$ are $\{(v,v_1,\ldots,v_{\ell},v'):(v,v_1,\ldots,v_{\ell},v') \in \mathcal{P}_{\ell+1} \}$. For  $\gamma= (v,v_1,\ldots,v_{\ell}) $ let $\phi(\gamma):=v_{\ell}$. Note that if $(S_n)_{n=0}^{\infty}$ is a SRW on  $\mathbb{T}_d $ (labeled as above) started from $(v)$ (which is the root) then $(\phi (S_n))_{n=0}^{\infty} $ is a SRW on $G $ started from $v$. Denote the law of $(S_n)_{n=0}^{\infty}$ by $\mathbb{P}_v$.   

 Fix some $\gamma:= (v,v_1,\ldots,v_{k}) \in \mathcal{P}_k $ such that $v_k=u$. Finally, observe that \[\Pr_v[T_{D_{k}}=T_u] \ge \mathbb{P}_{v}[T_{\mathcal{P}_{k}}= T_{\gamma}  ]=\frac{1}{|\mathcal{P}_k|}=\frac{1}{d(d-1)^{k-1} }.\]

We now prove that $\Pr_v[T_{D_{k}}=T_u] \le \frac{C(s,d)}{d(d-1)^{k-1} }$. We prove this  by induction on $s $. The base case $\mathrm{t}(B_{k})=0 $ is trivial (it holds with $C(1,d)=1$). Now consider the case that $\mathrm{t}(B_{k})=s>0 $. Let $z \in D_k $ be such that $\Pr_v[T_{D_{k}}=T_z] =\max_{u \in D_k }\Pr_v[T_{D_{k}}=T_u] $. For an edge $e:=\{x,y\} \in E $  let $G_e:=(V,E \setminus \{e\} )$ be the graph obtained by deleting $e$ from $G$. Let $H_e:=(V_{e},E_{e})$ be the graph obtained from $G_e$ by connecting $x$ (resp.~$y$)  to the root of a  $d$-ary tree\footnote{The root of a $d$-ary tree is of degree $d-1$.} $\mathcal{T}_x$ (resp.~$\mathcal{T}_y$). Denote the law of SRW on $H_e$ by $\mathrm{P}^{(e)}$. Let $D_i^{(e)}:=\{u \in V_{e} : \mathrm{dist}_{H_{e}}(u,v)=i \}$ and $B_{k}^{(e)}:=\cup_{i=0}^{k} D_i^{(e)}$.   We now show that there is some constant $K(s,d)$ and an edge  $e=\{x,y\}\in E$ belonging to some cycle in $B_k$ such that $x \in B_k,y \in B_{k-1} $ and
\begin{equation}
\label{e:Ksd}
\Pr_v[T_{D_{k}}=T_z] \le K(s,d)\mathrm{P}_v^{(e)}[T_{D_{k}^{(e)}}=T_z].
\end{equation}   
Once this is established, invoking the induction hypothesis concludes the induction step. 

Consider an arbitrary cycle in $B_k$ with at most one vertex in $D_k$. Let $x$ be the vertex of the cycle which maximizes $\Pr_x[T_{D_{k}}=T_z]$. Let $e=\{x,y\},e'=\{x,y'\}$ be the two edges of the cycle which are incident to $x$. Without loss of generality, let $e$ be the one through which $x$  is less likely to be reached. More precisely, assume that
 \begin{equation}
 \label{e:xy}
 \Pr_v[X_{T_{x}-1}=y,T_x \le T_{D_k}] \le \Pr_v[X_{T_{x}-1}=y',T_x \le T_{D_k}].\end{equation}
Also, by the choice of $x$ we have that
 \begin{equation}
 \label{e:xy2}
 \Pr_x[T_{D_k}=T_{z}] \ge \Pr_y[T_{D_k}=T_{z}].\end{equation}
Note that if $x \in D_k$ and $x \neq z $ then $\Pr_v[T_{D_{k}}=T_z] =\mathrm{P}_v^{(e)}[T_{D_{k}^{(e)}}=T_z] $. If $x=z$ then by \eqref{e:xy} $\Pr_v[T_{D_{k}}=T_z] \le 2 \mathrm{P}_v^{(e)}[T_{D_{k}^{(e)}}=T_z] $. Now consider the case that $x \notin D_k$. Denote $T_{x,y}:=\min \{T_x,T_y\} $ and $T_{x}^+:=\inf\{t>0:X_t=x \}$.  Observe that \begin{equation}
\label{e:Pe1}
\Pr_v[T_{D_{k}}=T_z<T_{x,y} ] = \mathrm{P}_v^{(e)}[T_{D_{k}^{(e)}}=T_z<T_{x,y} ]. \end{equation} Thus in order to conclude the proof of \eqref{e:Ksd} it remains only to show that \[\Pr_v[T_{D_{k}}=T_z>T_{x,y} ]  \le \tilde C(s,d) \mathrm{P}_v^{(e)}[T_{D_{k}^{(e)}}=T_z>T_{x,y} ]. \] By  \eqref{e:xy} we have that \[ \Pr_v[T_x < \min \{ T_{D_k},T_y\} ] \ge \Pr_v[T_{y}<T_x <  T_{D_k} ] \ge \frac{1}{d}\Pr_v[T_y <  T_{D_k} ].    \]
Thus $\Pr_v[T_x <  T_{D_k} ] \ge \frac{2}{ d} \Pr_v[T_y <  T_{D_k} ] $.
By  \eqref{e:xy2} we get that \[\Pr_v[T_x <  T_{D_k} =T_{z}]=\Pr_v[T_x <  T_{D_k} ]\Pr_x[  T_{D_k} =T_{z}] \ge \frac{2}{ d} \Pr_v[T_y <  T_{D_k} ]\Pr_y[  T_{D_k} =T_{z}]   \]
$=\frac{2}{d}\Pr_v[T_y <  T_{D_k} =T_{z}]$. Hence, there exists some constant $M(s,d)$ such that
 \begin{equation}
 \label{e:Pe2}
 \begin{split}
& \Pr_v[T_{D_{k}}=T_z>T_{x,y}] \le \Pr_v[T_{D_{k}}=T_z>T_{x}]+\Pr_v[T_{D_{k}}=T_z>T_{y}]   \\ & \le (1+\frac{d}{2})\Pr_v[T_{D_{k}}=T_z>T_{x}] \le (d+2)  \Pr_v[T_{D_{k}}=T_z, T_x < \min \{ T_{D_k},T_y\}]   
\\ &   \le M(s,d)  \Pr_v[T_x < \min \{ T_{D_k},T_y\}]\Pr_x[T_{D_k}=T_{z},\min \{T_x^{+},T_y\} >  T_{D_k}] 
\\ & \le M(s,d)\mathrm{P}_v^{(e)}[T_x < \min \{ T_{D_k^{(e)}},T_y\}]\mathrm{P}_x^{(e)}[T_{D_k^{(e)}}=T_{z}<  T_x^{+}] \\ & \le M(s,d) \mathrm{P}_v^{(e)}[T_{D_{k}^{(e)}}=T_z>T_{x,y} ],  \end{split} \end{equation}
 where in the second inequality we have used the fact that $\Pr_x[\min \{T_x^{+},T_y\} >  T_{D_k}] \ge c(s,d)$ for some constant $c(s,d)>0$\footnote{This could be proved by induction on $s$.}  and that by the choice of $x$ (namely, by \eqref{e:xy2}) we have that $\Pr_y[T_{D_k}=T_{z} \mid T_x>  T_{D_k}  ] \le \Pr_x[T_{D_k}=T_{z} ]=\Pr_x[T_{D_k}=T_{z} \mid T_x^{+}>  T_{D_k}  ]   $ and so \[\Pr_x[T_{D_k}=T_{z} \mid \min \{T_x^{+},T_y\} >  T_{D_k}  ] \ge  \Pr_x[T_{D_k}=T_{z} \mid T_x^{+}>  T_{D_k}  ]  = \Pr_x[T_{D_k}=T_{z} ].\]
 We leave the missing details as an exercise. Finally, combining \eqref{e:Pe1} and \eqref{e:Pe2} yields \eqref{e:Ksd}.        
\end{proof} 
\section{Proof of Theorem \ref{thm:main}}
\label{s:6}
Part (ii) was proven in \S~\ref{s:3}. Let $G_n=(V_n,E_n)$ be a sequence of non-bipartite, finite, connected, $d$-regular asymptotically one-sided Ramanujan graphs satisfying Assumption 1. Let  $r_n \to \infty$ be as in Assumption 1. Pick some $k=k_n \to \infty $ such that $k_n^2=o(r_n)$.  From this point on we often suppress the dependence on $n$ from our notation. Denote the transition matrix of SRW on $G$ (resp.~$G(k)$) by $P$ (resp.~$K$) and its stationary distribution by $\pi$ (resp.~$\pi_{G(k)}$). Let $A$ be an arbitrary set such that $\pi(A) \le \alpha=\alpha_n:=d^{-3k^2} $.    Denote $Q:=P^{k+2k^2} $. 

Before proceeding with the proof, we explain the choice of $k+2k^2$ in the definition of $Q$. In order to obtain an upper bound on $\la_{K}(A)$ we shall apply Proposition \ref{prop: comparison} with $P^t$ (for some $t$) and $K$ in the roles of $P^{(2)}$ and $P^{(1)}$ (respectively) from Proposition \ref{prop: comparison}. The obtained estimate is useful only when $t \ge c k^2 $. Heuristically, this is related to the fact that a SRW on a $d$-regular tree is much more likely to be at time $t$ at some given vertex of distance $O(\sqrt{t})$ from its starting point, than at some other given vertex at distance $\gg \sqrt{t}$ from its starting point (and we want $k = O(\sqrt{t})$).

\medskip

Recall that $\rho_d:=\frac{2\sqrt{d-1}}{d}$. Let $\la_{2}$ and $\la_2'$ be the second largest eigenvalues of $P$ and $Q$, respectively. Since $\la_2= \rho_d^{1-o(1)}$, by decreasing $k$ if necessary, we may assume that $ \la_2 \le \rho_{d}^{1-\frac{1}{3k^2 \log d}}$. By Proposition \ref{prop: laAla2} (using the notation from there) and our choice of $\alpha$,   \begin{equation}
\label{eq: laAla2}
 \la_Q(A) \le \la_2' + \alpha =\la_2^{k+2k^{2}}+ \alpha \le C_1 \rho_{d}^{k+2k^{2}}.\end{equation}  

 Let $(S_t)_{t=0}^{\infty}$ be SRW on $\mathbb{T}_d$,  the infinite $d$-regular tree rooted at $o$. 
 Denote its transition kernel by  $P_{\mathbb{T}_d}$. Denote the $i$th level of $\mathbb{T}_d$ by $\mathcal{L}_i $. Let $\tilde S_t$ be the level $S_t$ belongs to. Let $v \in  \mathcal{L}_k $. Let $T_{0}^+:=\inf\{t>0:\tilde S_t=0\}$. Then by Lemma \ref{lem:Z} (second inequality)
\begin{equation}
\label{e:Td1}
\begin{split}
& |\mathcal{L}_k|P_{\mathbb{T}_d}^{k+2k^2}(o,v)=\Pr_{0}[\tilde S_{k+2k^2} =k  ] \ge \Pr_{0}[\tilde S_{k+2k^2} =k,T_{0}^+>k+2k^2  ]  \\ & \ge c_{0} k^{-2} 2^{k+2k^{2}}(d-1)^{k^2+k-1}d^{-(k+2k^2)+1} \ge c_{1} k^{-2} (d-1)^{\frac{k}{2}}\rho_d^{2k^2+k}  
\end{split}
\end{equation}
 Let $x,y$ be a pair of adjacent vertices in $G(k)$.   It is standard that $P^{t}(x,y) \ge P_{\mathbb{T}_d}^{t}(o,v)$ for all $t$ (where $v$ is as above), and so by \eqref{e:Td1}
\begin{equation}
\label{e:Qxy}
Q(x,y)=P^{k+2k^2}(x,y) \ge P_{\mathbb{T}_d}^{k+2k^2}(o,v) \ge (d-1)^{\frac{k}{2}(1-o(1))}\rho_d^{2k^2+k}=:C_k. \end{equation}
 By Proposition \ref{prop: comparison}  (and borrowing the notation from there) in conjunction with \eqref{eq: laAla2}, \eqref{e:Qxy} and Assumption 1 (which implies that there exists some constant $C_{0}=C_0(d)>0$ such that  $L:=\frac{\max_x \deg_{G(k)}(x)}{\min_y \deg_{G(k)}(y)} \le C_0 $ and that if $x,y$ are of distance $k$ in $G$ then $K(x,y) \le C_{0}(d-1)^{-k}$), we have that \[\la_K(A) \le \la_Q(A)C_0^{3}(d-1)^{-k}/C_k=(d-1)^{-\frac{k}{2} (1-o(1))}.    \]
Denote the probability w.r.t.~SRW on $G(k)$ by $\mathbb{P}$. By \eqref{eq: spectralhit} we have for all $t$ (uniformly) that 
\begin{equation}
\label{eq: hitG(k)}
\max_{(a,A):a \in A, \pi(A) \le \alpha}\mathbb{P}_a[T_{A^c}>t] \le \sqrt{ C_0\alpha |V|} (d-1)^{-\frac{tk}{2}(1-o(1))}=\sqrt{ \alpha |V|} (d-1)^{-\frac{tk}{2}(1-o(1))} , \end{equation}
where we have used the fact that  $\max_{x \in V } \pi_{G(k)}(x)/\pi(x) \le C_0$, where $C_0$ is as above. 

Consider SRW on $G$, $(X_t)_{t=0}^{\infty}$.  Let $T_0:=0$ and inductively, $T_{i+1}:=\inf \{t \ge T_i : \mathrm{dist}_{G}(X_{T_{i+1}},X_{T_{i}})=k \} $. As in Definition \ref{d:Y}, consider the chain $\mathbf{Y}=(Y_{i})_{i=0}^{\infty}$, where $Y_{i}:=X_{T_{i}} $   for all $i$.  Let $W$ be its transition matrix. By Assumption 1 and  Lemma \ref{lem:W}  there exists some  constant $C=C(d)$ such that for all $x,y \in V $ of distance $k$ from one another (in $G$), 
\begin{equation}
\label{e:W2} 
1/C \le W(x,y)/K(x,y) \le C.
\end{equation} Denote the probability w.r.t.~$\mathbf{Y}$ by $\mathbf{P}$. Then by \eqref{eq: hitG(k)} and \eqref{e:W2} 
\begin{equation}
\label{eq: hitY}
\max_{(a,A):a \in A, \pi(A) \le \alpha}\mathbf{P}_a[T_{A^c}>t]  \le C^{t} \max_{(a,A):a \in A, \pi(A) \le \alpha}\mathbb{P}_a[T_{A^c}>t] \le \sqrt{ \alpha |V|} (d-1)^{-\frac{tk}{2}(1-o(1))} , \end{equation}
uniformly for all $t$. Denote the distribution of SRW on $G$ by $\Pr$. 
Observe that for all $s,t \ge 0$  \[\max_{(a,A):a \in A, \pi(A) \le \alpha}\mathrm{P}_a[T_{A^c}>  t+s] \le \max_{(a,A):a \in A, \pi(A) \le \alpha}\mathbf{P}_a[T_{A^c}>\tau(t)]+\max_{a \in V} \mathrm{P}_a[T_{\tau(t)}> t+s],  \]
where \[\tau(t):= \lceil \frac{(d-2)t}{dk} \rceil.\]
To conclude the proof (using  \eqref{eq: hitmix} in conjunction with Lemma \ref{lem: lower1}), we now show that (for some $o(1)$ terms) substituting above $t=\lceil (1+o(1)) \frac{d}{d-2}\log_{d-1}|V|  \rceil $ and $s=t/\sqrt{k} +t^{2/3}$ (the value $2/3$ in the exponent can be replaced by any number in $(1/2,1)$) yields $\max_{(a,A):a \in A, \pi(A) \le \alpha}\mathrm{P}_a[T_{A^c}>  t+s]=o(1)$. By \eqref{eq: hitY} it suffices to show that for this choice of $s$ and $t$ we have that $\max_{a \in V} \mathrm{P}_a[T_{\tau(t)}> t+s]=o(1)$. 

Fix $s$ and $t$ as above. We say that time $j  $ is \emph{good} if $X_j$ has  $d-1$ neighbors of greater distance from $X_{T_{i(j)}}$, where $i(j)$ is the index for which $j \in [T_{i(j)},T_{i(j)+1})$.
Let \[U_i:=|\{t \in [T_i,T_{i+1}):t \text{ is not good} \}| \quad \text{and} \quad U:=\sum_{i=0}^{\tau(t)} U_{i}.\]
By Assumption 1 we have that $\max_{v}\Pr_v[U_0> \ell ] \le C' e^{-c \ell} $ for all $\ell$, for some constants $c,C'>0$ (this is left as an exercise). By the Markov property, it follows that \[\max_{v}\Pr_v[U> \frac{t}{\sqrt{k}}]=o(1).\]  

Consider a coupling of the SRW on $G$ $(X_j)_{j=0}^{\infty}$ with the SRW on $\mathbb{T}_d$ started from its root $o$ $(S_j)_{j=0}^{\infty}$ in which if $j$ is the $\ell$th good time, then $\mathrm{dist}_G(X_{j+1},X_{T_{i(j)}})< \mathrm{dist}_G(X_{j},X_{T_{i(j)}}) $ iff  $\mathrm{dist}_{\mathbb{T}_d}(S_{\ell+1},o)< \mathrm{dist}_{\mathbb{T}_d}(S_{\ell},o) $ (unless $S_{\ell}=o$, but there is no harm in neglecting this possibility, as the number of returns to $o$ has a Geometric distribution). Using this coupling we get that for all $a \in V$  we have that
\[\mathrm{P}_a[T_{\tau(t)}> t+s] \le \Pr_{a}[U >\frac{t}{\sqrt{k}} ]+\max_{0 \le j \le \lceil \frac{t}{\sqrt{k}}\rceil} \Pr_{o}[S_{ t+s-j} \in \cup_{i=0}^{\tau(t)+j} \mathcal{L}_i  ]=o(1). \]
To see that $\max_{0 \le j \le \lceil \frac{t}{\sqrt{k}}\rceil} \Pr_{o}[S_{ t+s-j} \in \cup_{i=0}^{\tau(t)+j} \mathcal{L}_i  ]=o(1)$ use the fact that the distance of $S_{t+s-j} $ from $o$ is concentrated around $\frac{d-2}{d}(t+s-j) $ within a window whose length is of order $\sqrt{t} $    (c.f.~\cite{ram} (2.2)-(2.3) pg.~9) and that by our choice of $s$ we have that $\frac{d-2}{d}(t+s-j)-(\tau(t)+j) \gg \sqrt{t} $, for all $0 \le j \le  \lceil \frac{t}{\sqrt{k}}\rceil$. \qed

\begin{lemma}
\label{lem:Z}
Let $M$ be the number of paths of length $k+2k^2$ in $\mathbb{Z}$, starting from $0$, which end at $k$ and do not return to $0$. Then $M \ge c_0 2^{k+2k^2}/k^2$. 
\end{lemma}
\begin{proof}
Let $(Z_i)_{i=0}^{\infty}$ be a SRW on $\Z$. Let $T_0^+:=\inf \{t>0: Z_t=0 \}$. Then
\[\Pr_0[Z_{k+2k^2}=k,T_{0}^+>k+2k^2 ] \ge \Pr_0[ T_{0}^+> k+2k^2  \ge T_k ] \min_{0 \le i \le k^2} \Pr_k[T_0>2i,Z_{2i}=k] \ge c_0k^{-2},  \]
where we have used the fact that  $\Pr_0[ T_{0}^+> k+2k^2 \ge T_{k} ] \ge c_1 \Pr_0[ T_{0}^+>T_{k} ]=c_{1}/(2k)$ and that $\Pr_k[T_0>2i,Z_{2i}=k] \ge \Pr_k[T_{\{0,2k\}}>2i]  \Pr_k[Z_{2i}=k \mid T_{\{0,2k\}}>2i  ] \ge c_2 \cdot \frac{1}{2k}  $ for all
$i \le k^2 $.
\end{proof}
\section*{Acknowledgements}
The author is grateful to Nathana\"el Berestycki, Gady Kozma, Eyal Lubetzky, Yuval Peres, Justin Salez, Allan Sly and Perla Sousi  for useful discussions.

\nocite{}
\bibliographystyle{plain}
\bibliography{Ramanujan}

\vspace{2mm}

\end{document}